\documentclass[12pt]{article}
\usepackage{amsmath,amssymb,amsthm}
\usepackage[mathscr]{eucal}
\usepackage{dsfont}
\usepackage{accents}
\usepackage{color,graphics,srcltx}
\usepackage{graphicx}
\usepackage{subcaption}
\usepackage{amsfonts}
\usepackage{graphicx}
\usepackage{epstopdf}
\usepackage{float}
\usepackage{cleveref}

\newtheorem{theorem}{Theorem}[section]

\newtheorem{proposition}[theorem]{Proposition}
\newtheorem{corollary}[theorem]{Corollary}
\newtheorem{lemma}[theorem]{Lemma}
\newtheorem{example}[theorem]{Example}

\def\IC{{\mathbb{C}}}
\def\IR{{\mathbb{R}}}

\def\cG{{\cal G}}
\def\cA{{\cal A}}
\def\cU{{\cal U}}

\def\cF{{\cal F}}
\def\cF{{\mathcal F}}

\def\cS{{\mathcal S}}

\def\bA{{\bf A}}
\def\bB{{\bf B}}

\def\bF{{\bf F}}
\def\bV{{\bf V}}

\def\b0{{\bf 0}}
\def\tr{{\rm tr}\,}

\def\){{\right)}}
\def\({{\left(}}

\def\[{\left [}
\def\]{\right ]}
\def\Re{{\rm Re\,}}
\def\Im{{\rm Im\,}}
\def\diag{{\rm diag}}

\def\){ \right)}
\def\({\left(}

\def\conv{{\rm conv}}\def\co{{\conv}}

\def\C{{\mathbb{C}}}

\begin{document}
\openup .52\jot
\title{The generalized numerical range of a set of matrices}
\author{}

\author{Pan-Shun Lau, Chi-Kwong Li, Yiu-Tung Poon, and Nung-Sing Sze}
\date{}
\maketitle

\begin{abstract}
 For a given set of $n\times n$ matrices $\cF$, we study the union of the 
$C$-numerical ranges of the matrices in the set $\cF$, denoted by $W_C(\cF)$.
We obtain basic algebraic and topological properties of $W_C(\cF)$, and 
show that there are connections between the geometric properties of $W_C(\cF)$ and 
the algebraic  properties of $C$ and the matrices in $\cF$. 
Furthermore, we consider the starshapedness and convexity of the set $W_C(\cF)$.
In particular, we show that if $\cF$ is the convex hull of two matrices
such that $W_C(A)$ and $W_C(B)$ are convex, then the set $W_C(\cF)$
is star-shaped. We also investigate the extensions of the results to 
the joint $C$-numerical range of an $m$-tuple of matrices.
\end{abstract}



\section{Introduction}

Let $M_n$ be the set of all $n\times n$ complex matrices.
The numerical range of $A \in M_n$ is defined by
$$W(A) = \{x^*Ax: x\in \IC^n, x^*x = 1\},$$
which is a useful tool for studying matrices and operators; for example,
see \cite[Chapter 22]{Halmos} and \cite[Chapter 1]{HJ}. In particular, there is an interesting interplay 
between the geometrical properties of $W(A)$ and the algebraic 
properties of $A$.

In this paper, for a nonempty set $\cF$ of matrices in $M_n$, we consider   
$$W(\cF) = \cup \{ W(A): A \in \cF\}.$$
We show that there are also interesting connections
between the geometrical properties of 
$W(\cF)$ and the properties of the matrices in $\cF$.

The study has motivations from different topics. 
We mention two of them in the following.
The first one arises in the study of Crouzeix's conjecture
asserting that for any $A \in M_n$,
$$\|f(A)\| \le 2\max\{ |f(\mu)|: \mu \in W(A)\}$$
for any complex polynomial $f(z)$, where $\|A\|$ denotes 
the operator norm of $A$, see~\cite{MC}.
Instead of focusing on a single matrix $A \in M_n$, one may consider
a complex convex set $K$ and show that 
$$\|f(A)\| \le 2 \max\{ |f(\mu)|: \mu \in K\}$$
whenever $W(A) \subseteq K$.
One readily shows that this is equivalent to the Crouzeix's conjecture.
In fact, one may focus on the case when $K$ is a convex polygon (including 
interior) because  $W(A)$ can always be approximated by convex polygons
from inside or outside.

Another motivation comes  from quantum information science.
In quantum mechanics, a pure state in $M_n$ 
is a rank one orthogonal projection of the form $xx^*$ for some unit
vector $x \in \IC^n$, and a general state is a  density matrix, which is 
a convex combination of pure states.
For a measurement operator $A\in M_n$, which is usually Hermitian, 
the measurement of a state $P$ is computed by 
$\langle A, P\rangle = \tr(AP)$.
As a result, 
$$W(A) = 
\{ x^*Ax: x \in \IC^n, x^*x = 1\} = \{ \tr(Axx^*): x\in \IC^n, x^*x = 1\}$$
can be viewed as the set of all possible measurements on pure states
for a given measurement operator $A$. 
By the convexity of the numerical range, we have 
$$W(A) =  \{ \tr(Axx^*): x\in \IC^n, x^*x = 1\}
= \{\tr(AP): P \hbox{ is a general state}\}.$$
So, $W(A)$ actually contains all possible measurements on general states
for a given measurement operator $A$. If we consider all possible 
measurements under a set $\cF$ of measurement operators, then it is 
natural to study $W(\cF)$. In fact, if we know the set $W(\cF)$, we may deduce 
some properties about the measurement operators in $\cF$. 
For example, one can show that

(a) $W(\cF) = \{\mu\}$ if and only if $\cF = \{\mu I\}.$

(b) $W(\cF) \subseteq \IR$ if and only if all matrices in $\cF$ 
are Hermitian.

(c) $W(\cF) \subseteq [0, \infty)$ if and only if all matrices in 
$\cF$ are positive semi-definite.

\medskip
It is worth pointing out that if $A \in M_n$ is not Hermitian, one 
may consider the Hermitian decomposition $A = A_1 + iA_2$ such that 
$A_1, A_2$ are Hermitian, and identify $W(A)$ as the joint 
numerical range of $(A_1,A_2)$ defined by
$$W(A_1,A_2) = \{ (x^*A_1x, x^*A_2x): x \in \IC^n, x^*x = 1\} \subseteq \IR^2.$$
One can study the joint numerical range of $k$-tuple of 
Hermitian matrices $(A_1, \dots, A_k)$ defined by
$$W(A_1,\dots, A_k) = \{ (x^*A_1x, \dots, x^*A_kx): x \in \IC^n, x^*x = 1\} 
\subseteq \IR^k.$$
Accordingly, one may consider the joint numerical range $W(\cF)\subseteq \IR^k$ 
of a set $\cF$ of $k$-tuple of Hermitian matrices in $M_n$.

It turns out that we can study $W(\cF)$ under a more general setting.
For a matrix $C \in M_n$, the $C$-numerical range of $A \in M_n$
is defined by 
$$W_C(A) = \{\tr(CU^*AU): U \hbox{ is unitary }\},$$
which has been studied by many researchers in connection to different topics;
see \cite{Li} and its references.
In particular, the $C$-numerical range
is useful in the study of quantum control and quantum information;
for example, see \cite{Thomas}. 
Denote by $\{E_{11}, E_{12}, \dots,  E_{nn}\}$ the standard basis for $M_n$.
When $C = E_{11}$, $W_C(A)$ reduces to $W(A)$; if $C$ is a rank $m$ orthogonal
projection, then $W_C(A)$ reduces to the $m$th-numerical range of $A$;
see~\cite{Halmos,HJ,LP}. 
For a non-empty set $\cF$ of matrices in $M_n$, we consider
$$W_C(\cF) = \cup \{ W_C(A): A \in \cF\}.$$
In Section 2, we study the connection between the 
geometric properties of $W_C(\cF)$ and the properties of the set $\cF$.
In Sections 3 and 4, we study conditions for $W_C(\cF)$ to be star-shaped or convex.
In Section 5, we consider the joint $C$-numerical range of 
$(A_1, \dots, A_k) \in M_n^k$ defined by 
$$W_C(A_1, \dots, A_k) = \{\left(\tr(CU^*A_1U), \dots,\tr(CU^*A_kU)\right):
U \hbox{ is unitary} \} \subseteq \IC^k.$$
If $C, A_1, \dots, A_k$ are Hermitian matrices, then 
$W_C(A_1, \dots, A_k) \subseteq \IR^k$. 
One may see \cite{Li} for the background 
and references on the $C$-numerical range.

If $\mu \in \IC$ and  $C = \mu I$, then 
$$W_C(A) = \{\mu \tr A\} \quad \hbox{ and } \quad W_C(A_1, \dots, A_k) = 
\{\mu(\tr A_1, \dots, \tr A_k)\}.$$
We will always assume that $C$ is not a scalar matrix to avoid
trivial consideration.

For convenience of discussion, we always identify $\IC$ with $\IR^2$.

\section{Basic properties of $W_C(A)$ and $W_C(\cF)$}

We first list some basic results in the first two propositions below
about the $C$-numerical range;
see \cite{CT,Li,LT,NKTsing1984,Westwick1975} and their references. 

\begin{proposition}\label{P2.1}

Let $C, A \in M_n$.
\begin{enumerate}
\item[{\rm (a)}] For any unitary $U,V\in M_n$, $W_{V^*CV}(U^*AU)=W_C(A)= W_A(C)$.

\item[{\rm (b)}] For any $\alpha, \beta \in \IC$,
$$W_C(\alpha A + \beta I) = \alpha W_C(A) + \beta \tr C 
= \{ \alpha \mu + \beta \tr C: \mu \in W_C(A)\}.$$

\item[{\rm (c)}] The set $W_C(A)$ is compact.

\item[{\rm (d)}] The set $W_C(A)$ is star-shaped with 
$(\tr C)(\tr A)/n$ as a star center, i.e., 
$$t\mu+(1-t)(\tr C)(\tr A)/n\in W_C(A) \quad \hbox{
for any } \mu\in W_C(A) \hbox{  and }  t\in [0,1].$$

\item[{\rm (e)}] The set $W_C(A)$ is convex if 
there is $\gamma \in \IC$ such that $\tilde C = C - \gamma I_n$ satisfies any 
one of the following conditions.

{\rm (e.1)} $\tilde C$
is  rank one.

{\rm (e.2)} $\tilde C$ is a multiple of a Hermitian matrix.

{\rm (e.3)} $\tilde C$ is unitarily similar to a block matrix of the form
$(C_{ij})_{1\le i, j \le m}$
such 

\qquad that $C_{11}, \dots, C_{mm}$  are 
square matrices possibly with different sizes, and

\qquad $C_{ij} = 0$ if $j \ne i+1$.
\end{enumerate}

\end{proposition}

Researchers have extended the results on classical
numerical range, and showed that
there is interesting interplay between the geometry of $W_C(A)$ and 
the algebraic properties of $C$ and $A$. 
A boundary point $\nu$ of $W_C(A)$ is a corner point if 
there is $d > 0$ such that $W_C(A) \cap \{ \mu \in \IC: |\mu-\nu| \le d\}$
is contained in a pointed cone with $\nu$ as a vertex.

\begin{proposition} \label{P2.2}
Let $C = (c_{ij}) \in M_n$ be a non-scalar matrix in lower
triangular form with diagonal entries $c_1, \dots, c_n$. 
Suppose $A = (a_{ij}) \in M_n$ is also in lower triangular form with 
diagonal entries $a_1, \dots, a_n$.

\begin{enumerate}
\item[{\rm (a)}] If $\tr (CA) = \sum_{j=1}^n c_j a_j$ is a boundary point of 
$W_C(A)$, then $a_{rs} = 0$ whenever  $c_r \ne c_s$ 
and $c_{pq} = 0$ whenever $a_p \ne a_q$.
In particular, if $C$ and $A$ has distinct eigenvalues, then 
$C$ and $A$ will be in diagonal form, i.e., $C$ and $A$ are normal.

\item[{\rm (b)}] If $\nu$ is a corner point of $W_C(A)$, then there is a unitary $V$
such that $V^*AV$ is lower triangular form with diagonal entries 
$a_{j_1}, \dots, a_{j_n}$ such that
$(j_1, \dots, j_n)$ is a permutation of $(1, \dots, n)$ and   
$$\nu = \tr(CV^*AV) = \sum_{\ell=1}^n c_\ell a_{j_\ell}.$$

\item[{\rm (c)}] 
The set $W_C(A)$ is a singleton if and only if $A$ is a scalar matrix.

\item[{\rm (d)}] The set $W_C(A)$ is a non-degenerate line segment if and only if
$C$ and $A$ are non-scalar normal matrices having collinear eigenvalues in $\C$.

\item[{\rm (e)}] 
The set $W_C(A)$ is a convex polygon if and only if
$$W_C(A) = \conv \{ (c_1, \dots, c_n)P(a_1, \dots a_n)^t: P
\hbox{ is a permutation matrix}\}.$$
\end{enumerate}
\end{proposition}

Now, we extend the basic properties of $W_C(A)$ to $W_C(\cF)$.
We always assume that $C$ is not a scalar matrix and  $\cF$ is a non-empty subset of $M_n$.

\begin{theorem} \label{T3.1}
Suppose $C \in M_n$ is non-scalar and $\cF \subseteq M_n$ is non-empty.
\begin{enumerate}
\item[{\rm (a)}] If $U \in M_n$ is unitary, then
$W_C(\cF) = W_C(U^*\cF U),$ where $U^*\cF U = \{ U^*AU: A \in \cF \}.$

\item[{\rm (b)}] For any $\alpha, \beta \in \IC$, let
$\alpha \cF + \beta I = \{\alpha A + \beta I: A \in \cF\}$.
Then
$$W_C(\alpha \cF + \beta I) = \alpha W_C(\cF) + \beta \tr C
= \{ \alpha \mu + \beta\tr C : \mu \in W_C(\cF)\}.$$ 

\item[{\rm (c)}] If $\cF$ is bounded, then so is $W_C(\cF)$.
\item[{\rm (d)}] If $\cF$ is connected, then so is $W_C(\cF)$.
\item[{\rm (e)}] If $\cF$ is compact, then so is $W_C(\cF)$.
\end{enumerate}
\end{theorem}

\begin{proof}
(a) and (b) can be verified readily.

(c) If $\cF$ is bounded so that there is $R > 0$ such that 
for every $A \in M_n$ we have $\|A\| < R$, then 
$$|\tr(CU^*AU)| \le n\|C\| \|A\| < n\|C\| R.$$
Thus, $W_C(\cF)$ is bounded.

(d)
Note that for any $A \in M_n$, 
$W_C(A)$ is star-shaped with $(\tr C)(\tr A)/n$ as a star center.
If $\mu_1 = \tr(CU^*A_1U)$ and $\mu_2 = \tr(CV^*A_2V)$ with 
$A_1, A_2 \in \cF$ and $U,V$ unitary, then there are a line segment with end points
$\mu_1$ and $(\tr C)(\tr A_1)/n$,
and another line segment with end points
$\mu_2$ and $(\tr C)(\tr A_2)/n$.
If $\cF$ is connected, then so are the sets $\{\tr A: A \in \cF\}$ 
and $\{(\tr A)(\tr C)/n: A \in \cF\}.$
Hence, there is a path joining $\mu_1$ to $(\tr C)(\tr A_1)/n$,  then 
to $(\tr C)(\tr A_2)/n$ and then to $\mu_2$.

(e) Suppose $\cF$ is compact. Then $\cF$ is bounded and closed.
By (c), $W_C(\cF)$ is also bounded. To show that $W_C(\cF)$ is 
closed, let  $\{\tr(CU_k^*A_kU_k): k= 1,2,\dots\}$ be a  sequence
in $W_C(\cF)$ converging to $\mu_0 \in \IC$, 
where $A_k \in \cF$ and $U_k$ is unitary for each $k$.
Since $\cF$ is compact,  
there is a subsequence $\{A_{j_k}: k = 1,2,\dots\}$
of $\{A_k: k = 1,2,\dots\}$ converging to $A_0 \in \cF$. 
We can further consider a subsequence 
$\{U_{j(\ell)}: \ell = 1,2,\dots\}$ of 
$\{U_{j_k}: k = 1,2,\dots\}$ converging to $U_0$. Thus 
$$\{ \tr(CU_{j(\ell)}^* A_{j(\ell)} U_{j(\ell)}): k = 1,2,\dots\} \rightarrow
\tr(CU_0^*A_0U_0) = \mu_0\in W_C(\cF)\,.$$
Thus, $W_C(\cF)$ is closed. As a result, $W_C(\cF)$ is compact.  
\end{proof}

The following examples show that
none of the converses of the assertions in 
\Cref{T3.1} (c) -- (e) is valid, and there are no implications between 
the conditions that ``$\cF$ is closed'' and 
``$W_C(\cF)$'' is closed.

\begin{example} \label{E2.3} \rm
\begin{enumerate} 
\item[(a)] Suppose $C\in M_n$ is non-scalar
and has trace zero, and $\cF = \{\mu I: \mu \in \IC\}$.
Then $W_C(\cF) = \{0\}$ is bounded and compact, 
but $\cF$ is not bounded.

\item[(b)] Suppose $C \in M_n$ is non-scalar,
and $\cF = \{A_0, A_1\}$ such that
$A_0 = 0$ and  
$A_1 = xy^*$ for a pair of orthonormal vectors $x, y$.
Then $W_C(A_0) = \{0\}$ and 
 $W_C(\cF) = W_C(A_1)$ is a circular disk center at the origin with radius 
$$R = \max\{ |u^*Cv|:  \{u,v\} \hbox{ is an orthonormal set} \}.$$
Thus, $W_C(\cF)$ is connected,
but $\cF$ is not.

\item[(c)] Let $\cF = \conv \cG$ with
$$\cG = \{2E_{12}\} \cup \{\diag(e^{ir}, e^{-ir}): r \hbox{ is a
rational number}\}.$$
Then $\cF$ is not closed but $W(\cF) = W(2E_{12})$ is closed.

\item[(d)]  Let $\cF =\left \{\diag\(0,x+\dfrac{i}{x}\):x>0\right\}\cup\{\diag(0,0)\}$. Then $\cF$ is closed, but 
$$W(\cF)=\{x+iy:x,\ y>0,\ xy\le 1\}\cup \{0\}$$
 is not closed. 
\end{enumerate}
\end{example}

Next, we consider the connection between the geometrical properties of 
$W_C(\cF)$ and the properties of $C$ and $\cF$.
Note that for any subset $\cS$ of $\IC$, if $\tr C \ne 0$ and $\cF = \{\mu I/\tr C : \mu \in \cS\}$, then we have 
$W_C(\cF) = \cS$.
Thus,  the geometrical shape of $W_C(\cF)$ may be quite arbitrary.
Also, if $C = \mu I$ is a scalar matrix, then 
$W_C(\cF) = \{\mu \tr A: A \in \cF\}$. Again, $W_C(\cF)$ does not 
contain much information about the matrices in $\cF$. 
Nevertheless, we have the following.

\begin{theorem} Suppose $C \in M_n$ is non-scalar, 
and $\cF \subseteq M_n$ is non-empty.
The following conditions hold.
\begin{enumerate}
\item[{\rm (a)}] The set $W_C(\cF) = \{\mu\}$ 
if and only if $\cF = \{\nu I:\nu \tr C = \mu\}$.

\item[{\rm (b)}] The set $W_C(\cF)$ is a subset of a straight line $L$
if and only if 

{\rm (i)} $\cF \subseteq\{\nu I: \nu \in \IC, \nu \tr C \in L\}$, or 

{\rm (ii)} there are complex units $\alpha, \gamma \in \IC$ such that

\quad  $\gamma(C-(\tr C) I/n)$ and $\alpha(A-(\tr A)I/n)$ 
are Hermitian 
for all $A \in \cF$,

\quad and $\{(\tr C)(\tr A)/n: A\in\cF\}$ is collinear.

\item[{\rm (c)}] The set $W_C(\cF)$ is a convex
polygon if and only if $W_C(\cF)
=\co\{v_1,\dots,v_m\}$ where each $v_j$ is of
the forms $(c_1, \dots, c_n)(a_1, \dots, a_n)^t$.
Here $c_1, \dots, c_n$ are eigenvalues of $C$, and 
$a_1, \dots, a_n$ are eigenvalues of some $A_j \in\cF$.

\end{enumerate}
\end{theorem}

\begin{proof}
Condition (a) follows from the 
fact that $W_C(A) = \{\mu\}$ 
if and only if $A= \nu I$ with $\nu \tr C = \mu$.

(b) Suppose $W_C(\cF) \subseteq L$. If $\cF \subseteq \{\mu I: \mu \in \IC\}$,  
then clearly $\cF \subseteq \{\nu I: \nu \tr C \in L\}.$ 
Let $\cF$ contains a non-scalar matrix $A$.
Then $W_C(A)$ must be a non-degenerate line segment contained in $L$. 
Therefore $C,A$ are normal with collinear eigenvalues in $\IC$, see \cite[(7.3)]{Li}. There exist complex units $\alpha, \gamma \in \IC$ such that
$\gamma(C-(\tr C) I/n)$ and $\alpha(A-(\tr A)I/n)$ are Hermitian. If $B\in\cF$ is a scalar matrix, then $\alpha(B-(\tr B)I/n)=0$ which is Hermitian. Now assume $B\in\cF$ is non-scalar. 
Let $\tilde C = \gamma(C-(\tr C) I/n)$. Then
$$W_{\tilde C}(\alpha(B-(\tr B)I/n))=
\gamma\alpha W_C(B)+\mu_B\subseteq\{\gamma\alpha z+\mu_B:z\in L \},$$ 
for some constant $\mu_B\in\IC$, and   
$$  W_{\tilde C}(\alpha(A-(\tr A)I/n))=
\gamma\alpha W_C(A)+\mu_A\subseteq\{\gamma\alpha z+\mu_A:z\in L \},$$ 
for some constant $\mu_A\in\IC$. Hence $W_{\tilde C}(\alpha(B-(\tr B)I/n))$ is a subset of a 
line segment parallel to $W_{\tilde C}(\alpha(A-(\tr A)I/n))\subseteq \IR$. As 
$0\in  W_{\tilde C}(\alpha(B-(\tr B)I/n))$, we have 
$ W_{\tilde C}(\alpha(B-(\tr B)I/n))\subseteq \IR$. 
Therefore, $\alpha(B-(\tr B)I/n)$ is Hermitian. The last assertion follows from 
$\{(\tr C)(\tr A)/n: A\in\cF\}\subseteq W_C(\cF)\subseteq L$. 
The sufficiency can be verified readily.

(c) Suppose $W_C(\cF) = \conv\{ v_1, \dots, v_m\}$ is a convex polygon.
Then for every $v_j$, there is $A_j \in \cF$ such that 
$v_j = \tr(CU_j^*A_jU_j)\in W_C(A_j)$ for some unitary $U_j \in M_n$.
Since $W_C(A_j) \subseteq W_C(\cF)$, we see that $v_j$ is a vertex point of $W_C(A_j)$. 
It follows that $v_j$ has the form $(c_1, \dots, c_n)(a_1, \dots, a_n)^t$,
where $c_1, \dots, c_n$ are eigenvalues of $C$ and $a_1, \dots, a_n$ are 
eigenvalues of $A_j$ arranged in some suitable order. The converse of the
assertion is clear.
\end{proof}

\section{Star-shapedness and Convexity}

In this section, we study the star-shapedness and convexity of $W_C(\cF)$.
If $\cF$ is not connected,  then $W_C(\cF)$  may not be connected so that 
$W_C(\cF)$ is not star-shaped or convex.
One might 
hope that if $\cF$ is star-shaped or convex, then $W_C(\cF)$ will inherit the 
properties.
However, the following examples show that $W_C(\cF)$ may fail to be convex (star-shaped, resp.) even if $\cF$ is convex (star-shaped, resp.).

 
\begin{example}\label{exam24}
Let $C=E_{11}$, $A = \diag(1+i, 1-i)$ and $\cF= \conv\{A, -A\}$.
Then
$$W_C(\cF)=W(\cF) = \bigcup_{t \in [0,1]} W(tA + (1-t)(-A)) =  
\bigcup_{s \in [-1,1]} sW(A).$$
As $W(A)=\conv\{1+i,1-i\}$, we have $W_C(\cF)=\conv\{0,1+i,1-i\} \cup \conv \{0, -1-i, -1+i\}$ which is not convex, see \Cref{fig1}.

\begin{figure}[h!]
\centering
\includegraphics[height=3.3cm]{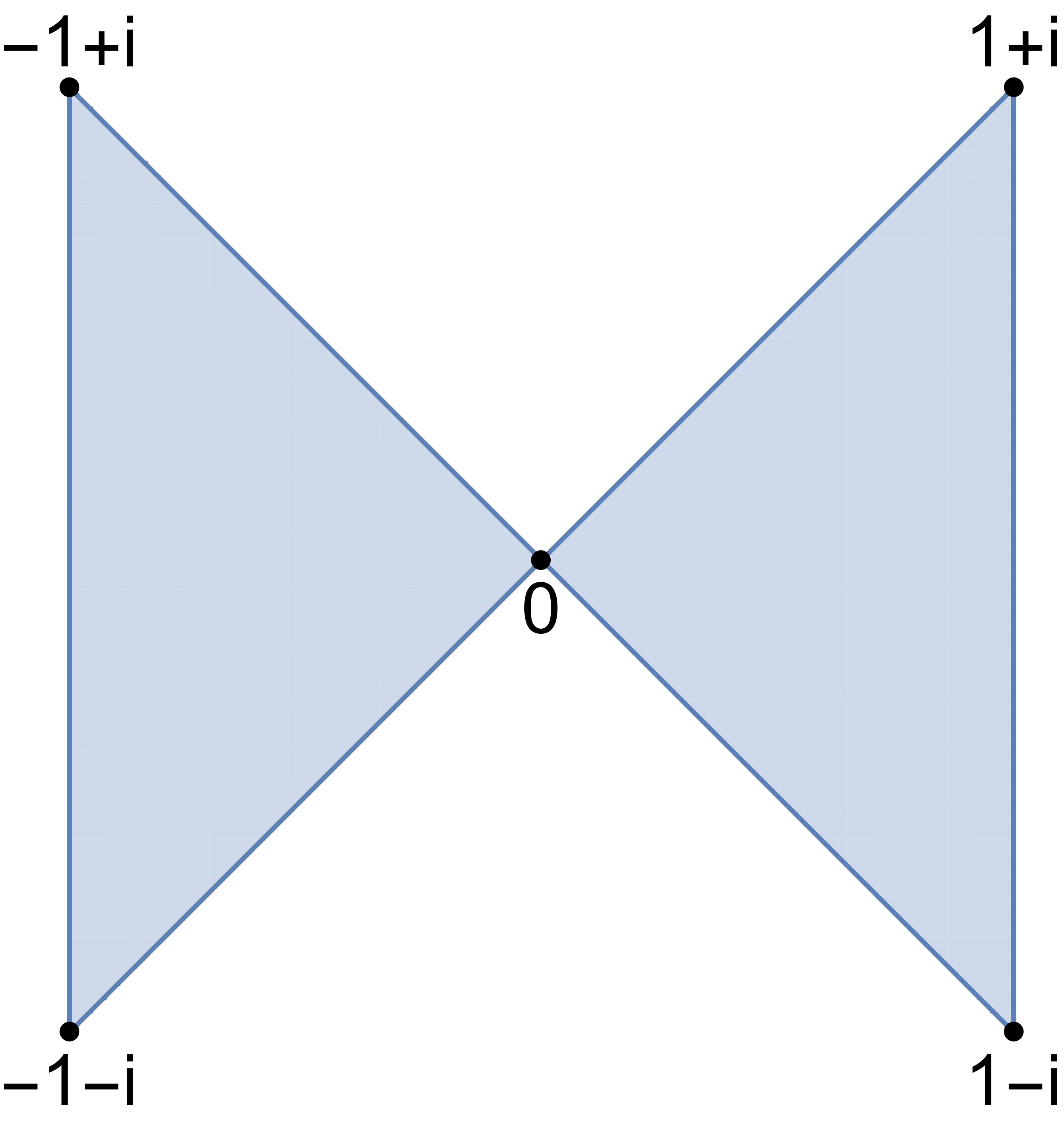}
\caption{}
\label{fig1}
\end{figure}
\vspace{-.7cm}
\end{example}

\begin{example}\label{exam25}
Let $C=E_{11}$, $A = \diag(1+i, 1-i)$, $\cF_1= \conv\{A, -A\}$ and $\cF_2=\conv\{A, -A+4I\}$.
Then $\cF=\cF_1\cup \cF_2$ is star-shaped with star-center $A$. Since $tA+(1-t)(-A+4I)=(1-2t)(2I-A)+2I$, we have
$$W(\cF_2) =\bigcup_{t \in [0,1]} W((1-2t)(-A+2I)+2I) =  
\bigcup_{s \in [-1,1]} sW(-A+2I)+2.$$
Note that $W(-A+2I)=\conv\{1+i,1-i\}=W(A)$. Then $W(\cF_2)=W(\cF_1)+2$ and
$$W_C(\cF)=W(\cF_1\cup \cF_2) = W(\cF_1)\cup W(\cF_2)$$
equals $\conv\{0,-1+i,-1-i\} \cup \conv \{0, 2, 1-i, 1+i\}\cup \conv\{2,3+i,3-i\}$ which is not star-shaped, see \Cref{fig2}. 

\begin{figure}[h!]
  \centering
    \includegraphics[height=3.3cm]{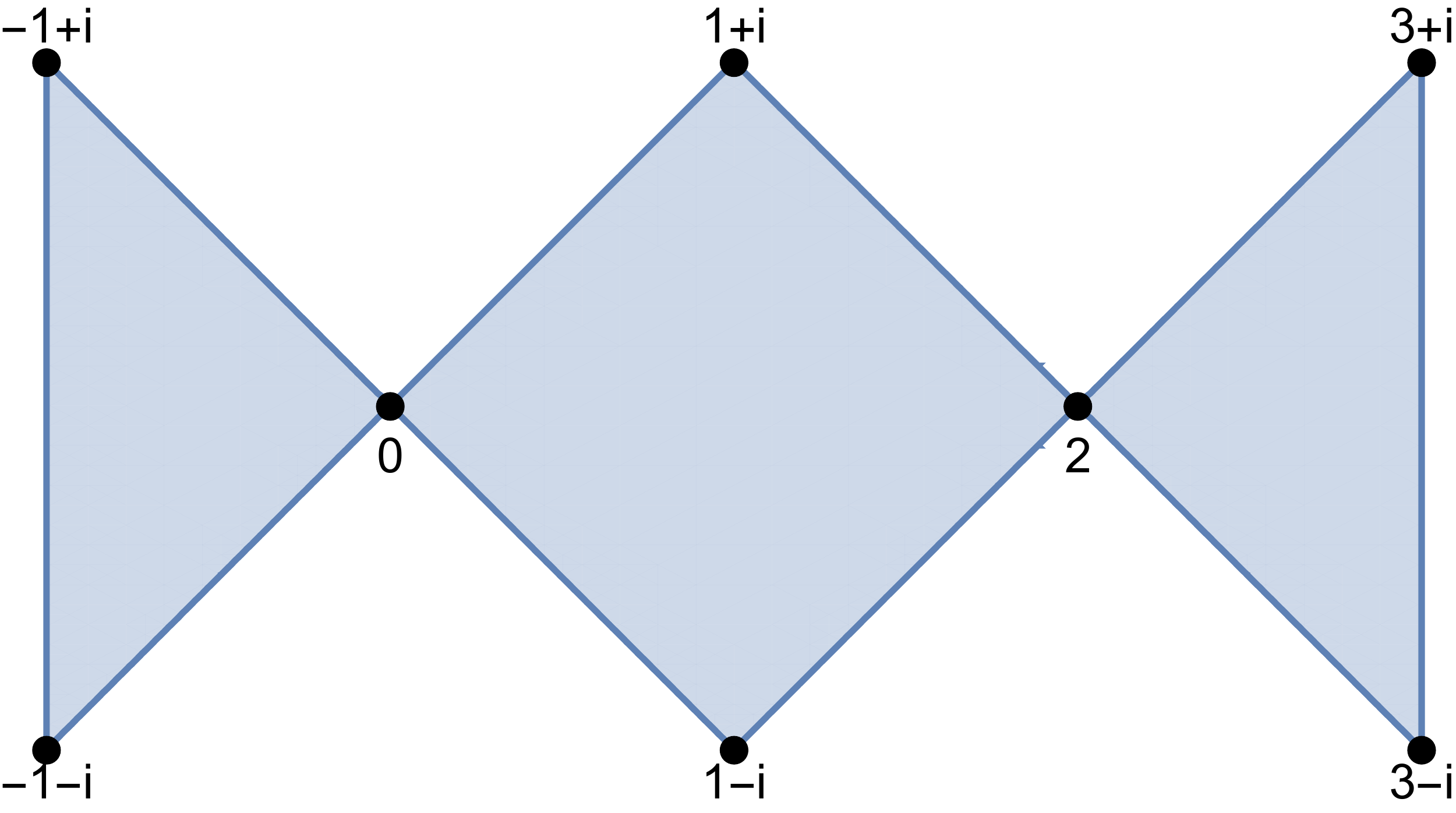} 
 \caption{} 
  \label{fig2}
\end{figure}
\end{example}

Notice that in \Cref{exam24}, $\cF$ is convex and $W_C(\cF)$ is star-shaped with 
star-center at the origin. One may ask if $W_C(\cF)$
is always star-shaped for a convex set $\cF$.
We will study this question in the following, and pay special attention on 
the case where $\cF = \conv\{A_1, \dots, A_m\}$
for some $A_1, \dots, A_m\in M_n$.

Denote by $S_C(A)$ the set of all star-centers 
of $W_C(A)$ and $\cU_n$ the group of all $n\times n$ unitary matrices. 
We begin with the following result showing that $W_C(\cF)$ is star-shaped if
$C$ or $\cF$ satisfies some special properties.

\begin{proposition} \label{2.1.2}
Suppose $C\in M_n$ and $\cF$ is a convex matrix set.
\begin{enumerate}
\item[ {\rm (a)}] If $\cF$ contains a scalar matrix $\mu I$, then
$W_C(\cF)$ is star-shaped with $\mu\tr C$ as a star center.
\item[ {\rm (b)}] Suppose
the intersection of all (or any three of)
$S_C(A)$ with $A \in \cF$ is nonempty. Then
$W_C(\cF)$ is star-shaped with $\mu$ as a star center for any
$\mu \in \cap\{S_C(A): A \in \cF\}$.
\item[ {\rm (c)}] If $\tr C=0$, then $W_C(\cF)$ 
is star-shaped with 0 as a star-center.
\item[ {\rm (d)}] If all matrices in $\cF$
have the same trace $\nu$, then $W_C(\cF)$ is star-shaped with 
$\nu \tr C$ as a star-center.
\end{enumerate}
\end{proposition}

\begin{proof} (a) Suppose $\cF$ contains a scalar matrix $\mu I$ and $B \in \cF$.
Then 
$$\conv\{\mu\tr C, W_C(B)\}\subseteq W_C(\conv\{\mu I, B\})\subseteq W_C(\cF).  $$
The result follows.

(b) Suppose $\mu \in \cap \{S_C(A): A \in \cF\}$.
Then for any $\nu \in W_C(\cF)$,
there is $B \in \cF$ such that $\nu \in W_C(B)$. As $\mu\in S_C(B)$, the
line segment joining $\mu$ and $\nu$ will lie in $W_C(B) \subseteq W_C(\cF)$.
Thus, $W_C(\cF)$ is star-shaped with $\mu$ as a star center.

If $S_C(A_0)\cap S_C(A_1) \cap S_C(A_2)\ne \emptyset$ for any
$A_0, A_1, A_2 \in \cF$, then $\cap \{S_C(A): A \in \cF\} \ne \emptyset$
by Helly's Theorem. So, the result follows from the preceding paragraph.

(c) Note that $\frac{1}{n}(\tr C)(\tr A)\in S_C(A)$ for any $C,A\in M_n$, see \cite{CT} or \Cref{P2.1} (d). If $\tr C=0$, then $0\in \bigcap\{ S_C(A):A\in\cF\}$ is a star-center of $W_C(\cF)$ by (b).

(d) The assumption implies that $\nu\tr C$ is the common star-center of 
$W_C(A)$ for all $A \in \cF$. Thus, the result follows from (b).
\end{proof}

\begin{lemma}\label{commoncenter2}
Let $C,A,B\in M_n$ and 
$\cF=\conv\{A,B\}$. If $\mu\in\ S_C(A)\cap S_C(B)$, 
then $W_C(\cF)$ is star-shaped with star-center $\mu$.
\end{lemma}
\begin{proof}
Let $\zeta\in W_C(\cF)$. There are $V\in\cU_n$ and $0\le t\le 1$ such that $\zeta=\tr(CV^*(tA+(1-t)B)V)$. It suffices to show $\co\{\mu, \tr(CV^*AV),\tr(CV^*BV)\}\subseteq  W_C(\cF)$.  For any $a,b,c\in\C$, we let $\triangle(a,b,c)=\conv\{a,b\}\cup \conv\{b,c\}\cup \conv\{a,c\}$, i.e., the triangle 
(without the interior) with the vertices $a,b,c$. Let $U_A\in\cU_n$ such that 
$ \tr(CU_A^*AU_A)=\mu.$ 
As $\mu\in\ S_C(A)\cap S_C(B)$ we have
$$\conv\{\tr(CV^*AV),\mu \}\cup\conv\{\tr(CV^*BV),\mu \}
\subseteq W_C(\cF).$$
Moreover we have
$$\conv\{\tr(CV^*AV),\tr(CV^*BV)\}=\{\tr(CV^*(tA+(1-t)B)V):0\leq t\leq 1\}\subseteq W_C(\cF).$$
Hence $\triangle(\tr(CV^*AV),\tr(CV^*BV),\mu)\subseteq W_C(\cF)$. We shall show that 
\begin{equation}\label{inclusion1}\conv\{\tr(CV^*AV),\tr(CV^*BV),\mu\}\subseteq W_C(\cF).\end{equation}
If $\triangle(\tr(CV^*AV),\tr(CV^*BV),\mu)$ is a line segment or a point, then \cref{inclusion1} holds clearly. Now assume that $\triangle(\tr(CV^*AV),\tr(CV^*BV),\mu)$ is non-degenerate. As $\cU_n$ is path-connected, we define a continuous function $f:[0,1]\to\cU_n$ with $f(0)=V$ and $f(1)=U_A$. For $0\le t\le 1$, let $V_A(t)=\tr(Cf(t)^*Af(t))$ and $V_B(t)=\tr(Cf(t)^*Bf(t))$.
Note that for any $t\in[0,1]$, we have
$$\triangle(t)=\triangle(V_A(t),V_B(t),\mu)\subseteq W_C(\cF).$$ 

\begin{figure}[H]
\centering
\includegraphics[height=5.5cm]{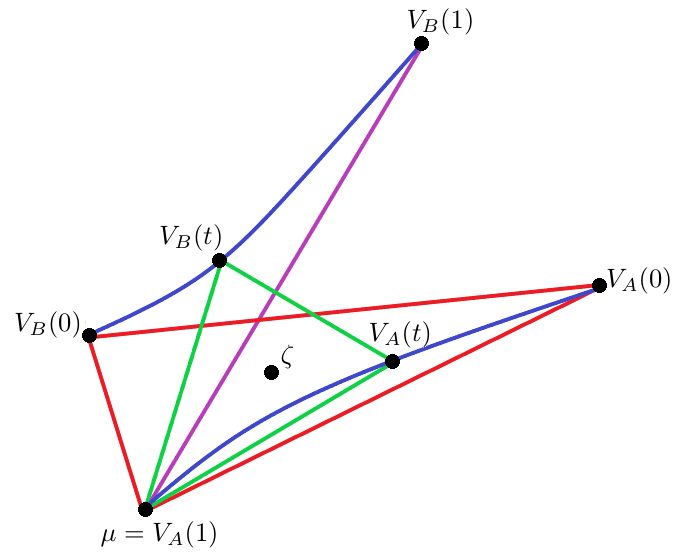}
\caption{$\triangle(0)$ and $\triangle(t)$ are triangles in the red and green, respectively. In particular, $\triangle(1)$ dengenerates to a line segment in purple.}
\label{fig3}
\end{figure}

For any $\zeta\in \conv \triangle(0)$, see \Cref{fig3}, let $$t_0=\max\{t:\zeta\in\conv\triangle(s) \mbox{ for all }0\le s\le t\}\,.$$
Since $\triangle(1)$ degenerates, by continuity of $f$, we have  $\zeta\in\triangle (t_0) \subseteq W_C(\cF)$. Hence the result follows. 
\end{proof}

\medskip
One can extend \Cref{commoncenter2} to a more general situation.
\begin{theorem}\label{intersect_m}
Let $C\in M_n$ and $\cG$ be a (finite or infinite) family of matrices in $M_n$. If $\mu\in\bigcap_{A\in\cG} S_C(A)$, then $W_C(\cF)$ is star-shaped with star-center $\mu$ for $\cF=\co\cG$.
\end{theorem}
\begin{proof}
By \Cref{commoncenter2}, the result holds if $\cG$ has 2 elements. Suppose $|\cG|\geq 3$ and $\mu\in\bigcap_{A\in\cG} S_c(A)$. Let $\zeta\in W_C(\cF)$. Then there exist $A_1,...,A_m\in\cG$, $t_1,...,t_m> 0$ with $t_1+\cdots+t_m=1$ and $U\in\cU_n$ such that $\zeta=\tr(CU^*(t_1A_1+\cdots +t_mA_m)U)$. Let $\zeta_i=\tr(CU^*A_iU),$ $i=1,...,m$. Then $\zeta\in\co\{\zeta_1,...,\zeta_m\}.$ The half line from $\mu$ through $\zeta$ intersects a line segment joining some $\zeta_i,\zeta_j$ with $1\leq i\leq j\leq m$ such that $\zeta\in\co\{\mu,\zeta_i,\zeta_j\}$. By \Cref{commoncenter2}, we have $\co\{\mu,\zeta_i,\zeta_j\}\subseteq W_C(\co\{A_i,A_j\})\subseteq W_C(\cF).$
\end{proof}

Note that if $W_C(A)$ is convex, then $W_C(A) = S_C(A)$. 
So, if $W_C(A)$ is convex for every
$A \in \cG$, and if $\mu \in \cap_{A \in \cG}W_C(A)$, then $\mu$ is a star-center 
of $W_C(\co(\cG))$. Checking the condition that $W_C(A)$ is convex for every $A \in \cG$
may not be easy. On the other hand, if $C \in M_n$ is such that
$W_C(A)$ is always convex, then one can skip the checking process,
and we have the following.

\begin{corollary}
Suppose $C\in M_n$ satisfies any one of the conditions {\rm (e.1)---(e.3)} in 
\Cref{P2.1} {\rm (e)}. If $\cG \subseteq M_n$ and 
$\mu \in \cap_{A \in \cG}W_C(A)$, then $\mu$ is a star-center of $W_C(\conv\cG)$.
\end{corollary}

Next, we show that if $W_C(A)$ and $W_C(B)$ are convex, then 
$W_C(\cF)$  is star-shaped for
$\cF=\conv\{A,B\}$ even when $W_C(A)\cap W_C(B) = \emptyset$.

\begin{theorem}\label{nonintersect}
Let $C\in M_n$. Suppose $A,B\in M_n$ such that $W_C(A)$ and $W_C(B)$ 
are convex sets with empty intersection.  Let $\cF=\co\{A,B\}$.

If $W_C(A)\cup W_C(B)$ lies on a line, then $W_C(\cF)=\co\{W_C(A),W_C(B)\}$ is convex.
 
Otherwise, there are two non-parallel lines $L_1$ and $L_2$ intersecting at $\mu$
such that for each $j = 1,2$, $L_j$ is a common supporting line of $W_C(A)$ and 
$W_C(B)$ separating the two convex sets (i.e., 
$W_C(A)$ and $W_C(B)$ lying on opposite closed half spaces 
determined by $L_j$); the set
$W_C(\cF)$ is star-shaped with star-center $\mu$, see \Cref{fig4}.

\begin{figure}[H]
\centering
{\includegraphics[height=3.5cm]{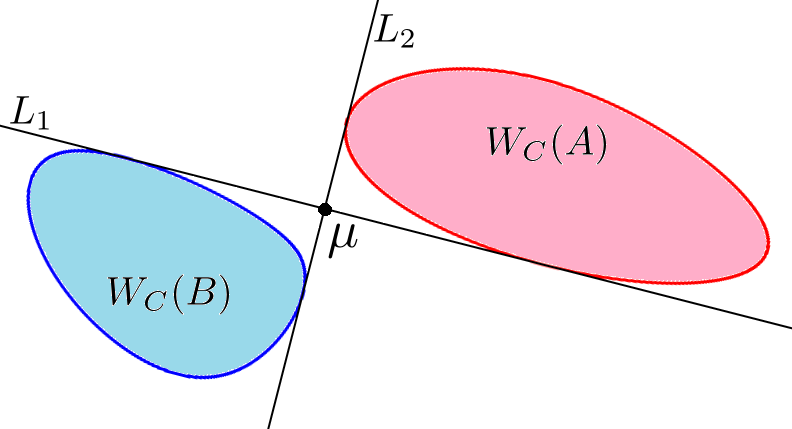} }
\caption{}
\label{fig4}
\end{figure}
\end{theorem}
\begin{proof}
Assume $W_C(A)\cup W_C(B)$ lies on a line. As $\cF$ is connected, by Theorem~\ref{T3.1}, $W_C(\cF)$ is connected. We have
\begin{eqnarray*} W_C(A)\cup W_C(B)&\subseteq &W_C(\cF)\subseteq\co\{W_C(A),W_C(B)\}.\end{eqnarray*}
Therefore, $W_C(\cF)= \co\{W_C(A),W_C(B)\}$ which is convex.

Otherwise, we may assume that $W_C(A)$ lies on the left open half plane and $W_C(B)$ 
lies on the right open half plane. For $t\in[-\pi,\pi]$, let $L(t)$ be the support line 
of $W_C(A)$ with (outward pointing) normal $n(t)=(\cos t,\sin t)$. Let 
$P(t)=L(t)\cap W_C(A)$. Therefore, for each $t\in[-\pi,\pi]$ and $p\in P(t)$, we have 
$n(t)\cdot(w-p)\leq 0$ for all $w\in W_C(A)$. Let 
$q(t)=\min\{n(t)\cdot(w-p):w\in W_C(B),p\in P(t)\}.$ 
Then we have $q(0)>0$ and $q(-\pi)=q(\pi)<0$. Hence, there exist $-\pi<t_1<0<t_2<\pi$ 
such that $q(t_1)=q(t_2)=0$. Then $L_j=L(t_j)$ for $j=1,2$ will satisfy the requirement. 
		
We claim that $L(t_1)$ and $L(t_2)$ in (a) cannot be parallel. Otherwise, we must 
have $t_1=t_2+\pi$. Therefore, $n(t_1)=-n(t_2)$ and $W_C(A),W_C(B)\subset L_1=L_2$, 
is a  contradiction. Since $L_1$ and $L_2$ are not parallel, they intersect at 
a point $\mu$.

Now, we show that $\mu$ is a star-center of $W_C(\cF)$.
We may apply a suitable affine transform to $A$ and $B$, 
and assume that $\mu=0$, $L_1$ and $L_2$ are the $x$-axis and $y$-axis respectively. 
We may further assume that $W_C(A)$ lies in the first quadrant and $W_C(B)$ lies in the 
third quadrant. 
For any $\zeta\in W_C(\cF)$, there are $V\in\cU_n$ and $0\leq t\leq 1$
such that $\zeta=\tr(CV^*(tA+(1-t)B)V)$. Denote $\zeta_A=\tr(CV^*AV)$ and $\zeta_B=\tr(CV^*BV)$. We claim that $\conv\{\zeta_A,\zeta_B,0\}\subseteq W_C(\cF)$. Once the claim holds, we have $\{s\zeta+(1-s)\cdot 0:0\leq s\leq 1\}\subseteq \conv\{\zeta_A,\zeta_B,0\}\subseteq  W_C(\cF)$. Then the star-shapedness of $W_C(\cF)$ follows.

We now show the claim. We denote by $\overline{\zeta_1\zeta_2}$ the line segment with end points $\zeta_1,\zeta_2\in\C$. By symmetry, we may assume 
that the line segment $\overline{\zeta_A\zeta_B}$ intersects the $y$-axis at $(0,b)$ 
with $b\geq 0$. The situation is depicted in \Cref{fig4:a} and \Cref{fig5:b}.

We shall first show that $\co\{\zeta_B,bi,0 \}\subseteq W_C(\cF)$. 
Let $y_B$ be a point in the intersection of $W_C(B)$ and the $y$-axis, $U_B\in\cU_n$ 
be such that $y_B=\tr(CU_B^*BU_B)$. Let $y_A=\tr(CU_B^*AU_B)$. Then by the convexity of 
$W_C(A)$, $W_C(B)$ and $\cF$, we have
\begin{equation*} Q((\zeta_A,y_A);(\zeta_B,y_B)):=\overline{\zeta_Ay_A}\cup
\overline{\zeta_By_B}\cup \overline{\zeta_A\zeta_B}\cup\overline{y_Ay_B}\subseteq 
W_C(\cF).\end{equation*}

\begin{figure}[H]
\begin{subfigure}{0.47\textwidth}
\includegraphics[height=1.2in]{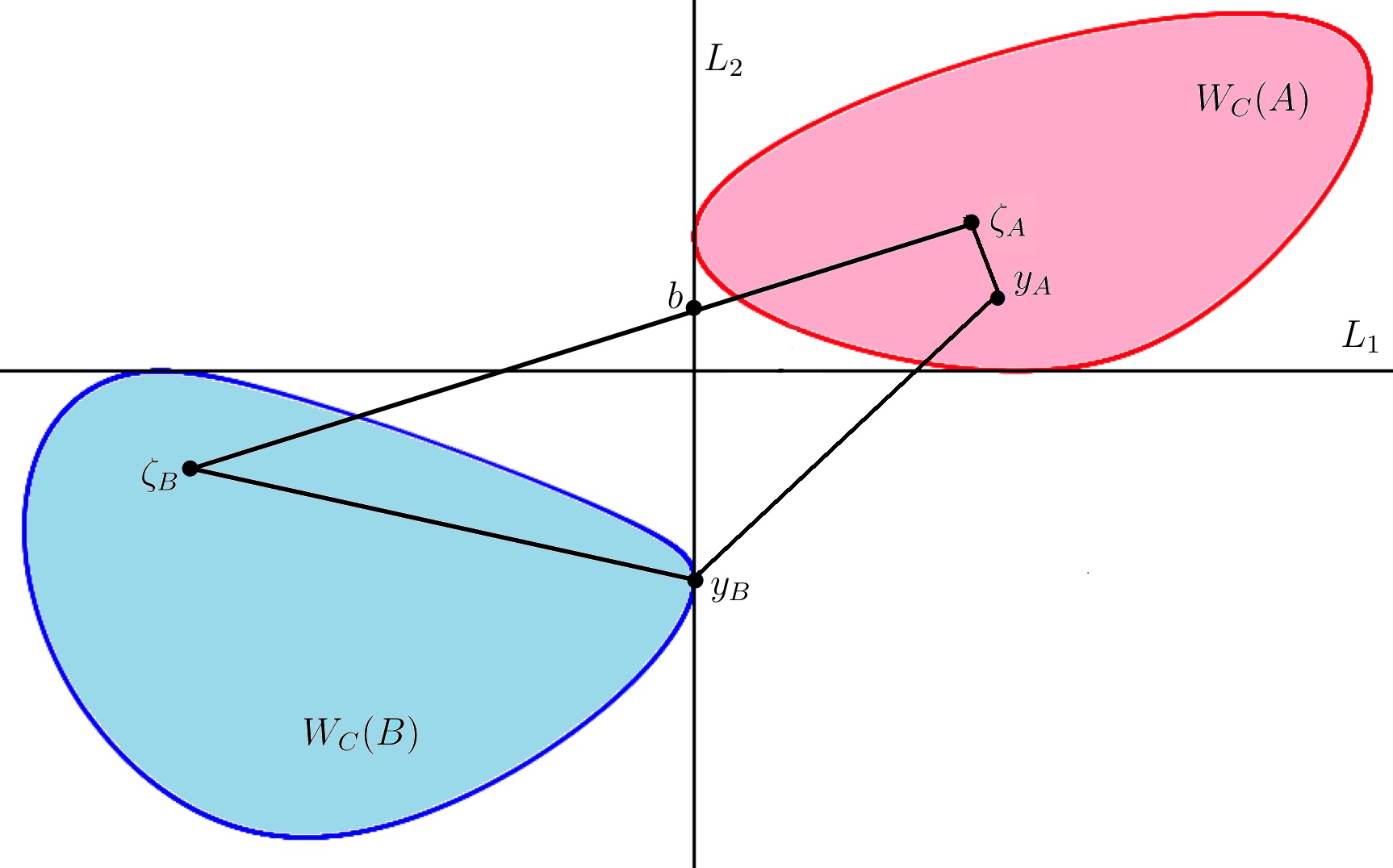}
\caption{$Q((\zeta_A,y_A);(\zeta_B,y_B))$ is a quadrilateral.}
\label{fig4:a}
\end{subfigure}\quad\quad
\begin{subfigure}{0.47\textwidth}
\includegraphics[height=1.2in]{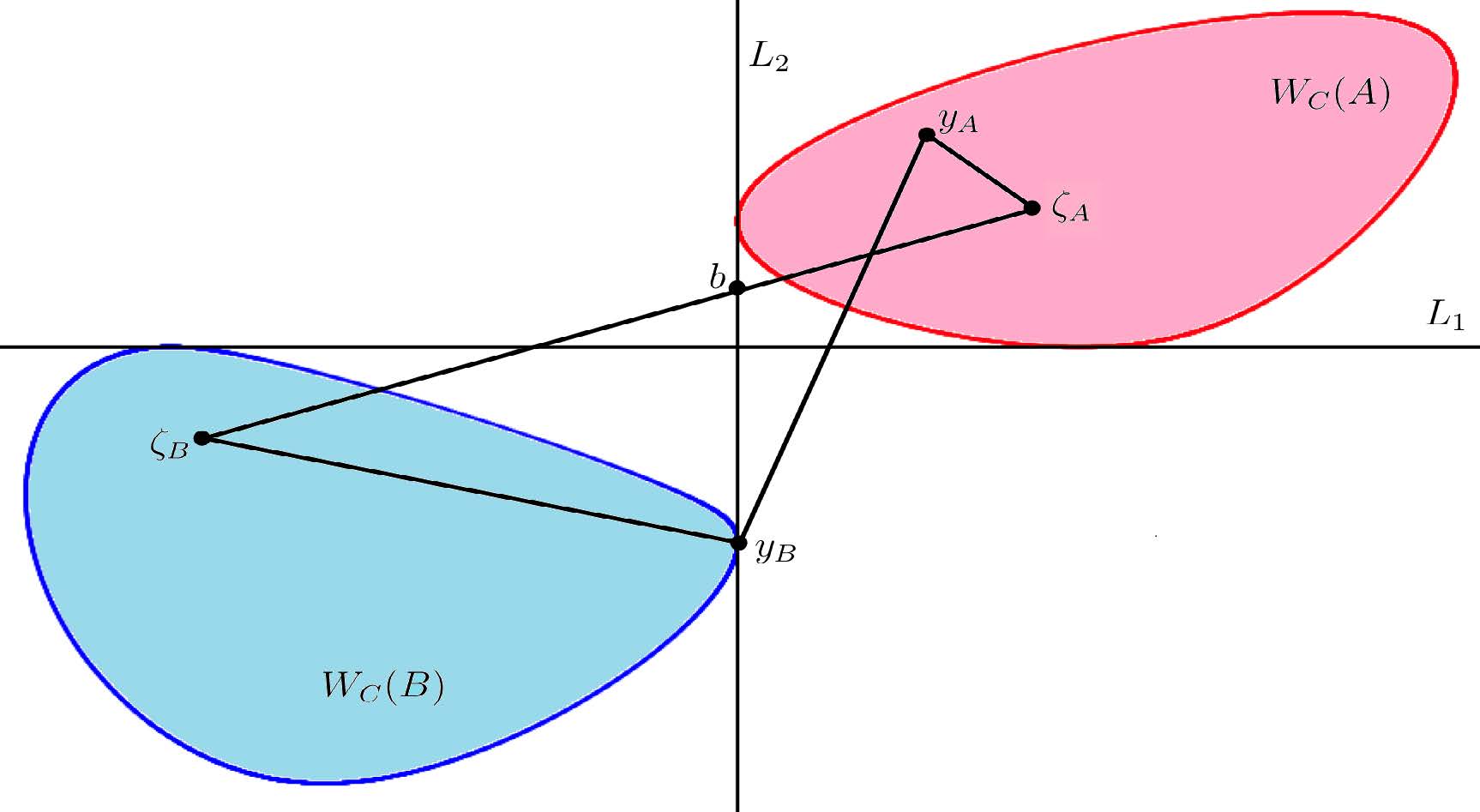}
\caption{$Q((\zeta_A,y_A);(\zeta_B,y_B))$ is a union of two triangle.}
\label{fig5:b}
\end{subfigure}
\caption{}
\end{figure}

Note that $Q((\zeta_A,y_A);(\zeta_B,y_B))$ is either a quadrilateral or a union of two triangle, see \Cref{fig4:a} and \Cref{fig5:b}.  In both case, $\triangle(\zeta_B,bi,0)$ lies inside the region determined by the closed curve 
$Q((\zeta_A,y_A);(\zeta_B,y_B))$. Now consider a continuous function $f:[0,1]\to\cU_n$ with $f(0)=V$ and $f(1)=U_B$. Then for $t\in[0,1]$, let $\zeta'_A(t)=\tr(Cf(t)^*Af(t))$ and $\zeta'_B(t)=\tr(Cf(t)^*Bf(t))$, we have
$$Q((\zeta'_A(t),y_A);(\zeta'_B(t),y_B)):=\overline{\zeta'_A(t)y_A}\cup\overline{\zeta'_B(t)y_B}\cup \overline{\zeta'_A(t)\zeta'_B(t)}\cup\overline{y_Ay_B}\subseteq W_C(\cF).$$
Note that $Q((\zeta'_A(1),y_A);(\zeta'_B(1),y_B))$ degenerates, by continuity of $f$, for any point $\zeta$ enclosed by the closed curve $Q((\zeta_A,y_A);(\zeta_B,y_B))$,  there exists $0\leq t_0\leq 1$ such that $\zeta\in Q((\zeta'_A(t_0),y_A);(\zeta'_B(t_0),y_B)) \subseteq W_C(\cF)$. Hence, $\co\{\zeta_B,bi,0\}\subseteq W_C(\cF)$. 
By symmetry, one can show that $\co\{\zeta_A,a,0\} \subseteq W_C(\cF)$ where $a$ is the intersection point of $\overline{\zeta_A\zeta_B}$ and $x$-axis. 
Hence, $\co\{\zeta_A,\zeta_B,0\}=\co\{\zeta_A,a,0\}\cup ~\co\{\zeta_B,bi,0\}\subseteq W_C(\cF)$. The claim follows.	
\end{proof}

\section{Additional results on star-shapedness and convexity}

It is not easy to extend \Cref{nonintersect}.
The following example shows that if $W_C(A)$ and $W_C(B)$ are not convex, then 
$W_C(\cF)$ may not be star-shaped for
$\cF=\conv\{A,B\}$ even when $W_C(A)\cap W_C(B) \ne \emptyset$.

\begin{example}
Let $w=e^{2\pi i/3}$ and $C=\diag(1,w,w^2)$. Suppose $A=C-\frac{1}{6}I$, $B=e^{\pi i/3} C+\frac{1}{6}I$ and $\cF=\co\{A,B\}$. 
Then $W_{C+I}\left(\cF\right)$ is not star-shaped.

\rm To prove our claim,
 for  
 $0\leq t\leq 1$,  let 
$$A(t)=tA+(1-t)B=(t+(1-t)e^{\pi i/3})C+\frac{1-2t}{6}I\,.$$
\begin{eqnarray*}
W_{C+I}(A(t))&=&W_{C+I}\left((t+(1-t)e^{\pi i/3})C+\frac{1-2t}{6}I\right)\\
&=&(t+(1-t)e^{\pi i/3})W_{C}\left(C\right)+\left(t+(1-t)e^{\pi i/3}+\frac{1-2t}{6}\right)\tr C+\frac{1-2t}{2}\\
&=&(t+(1-t)e^{\pi i/3})W_{C}\left(C\right)+\frac{1-2t}{2}.
\end{eqnarray*}

Therefore,
$$
W_{C+I}(\cF)=\bigcup \{(t+(1-t)e^{\pi i/3})W_{C}\left(C\right)+\frac{1-2t}{2}:0\le t\le 1\}\,.$$

By  \cite{Nakazato2006}, $W_{C}\left(C\right)$ is star-shaped with the origin as the unique star-center. Moreover $W_C(C)=e^{2\pi i/3}W_C(C)$ and its boundary is given by
$$\Gamma=\{ 2e^{i\theta}+e^{-2i\theta}:-\pi\leq \theta\leq \pi\}.$$

Hence, $W_{C+I}(A(t))$   is star-shaped with $\frac{1-2t}{2}$ as the unique star-center and  its boundary is given by
\begin{equation}\label{gamma}\Gamma(t)=\{(t+(1-t)e^{\pi i/3})\( 2e^{i\theta}+e^{-2i\theta}\):-\pi\leq \theta\leq \pi\}.
\end{equation}

For $0\le t\le 1$ and $-\pi\leq \theta\leq \pi$, define 

\begin{eqnarray*}
f(\theta,t)&=&(t+(1-t)e^{\pi i/3})(2e^{i\theta}+e^{-2i\theta})+\frac{1-2t}{2}\\
&=& \frac{1+t}{2}(2\cos\theta+\cos 2\theta)-\frac{1-t}{2}\sqrt{3}(2\sin\theta-\sin 2\theta)+\frac{1-2t}{2}\\
& &~+~i\left(\frac{1-t}{2}\sqrt{3}(2\cos\theta+\cos 2\theta) +\frac{1+t}{2} (2\sin\theta-\sin 2\theta)\right).\end{eqnarray*}

Using this description, we can plot $W_{C+I}(\cF)$ as follows.

First consider the plot of $\Gamma(0) $ and $\Gamma(1)$, which are the boundaries of $W_{C+I}(B)$ and $W_{C+I}(A)$ respectively, see \Cref{fig9}.

\smallskip
\begin{figure}[H]
\centering
\includegraphics[ height=4cm]{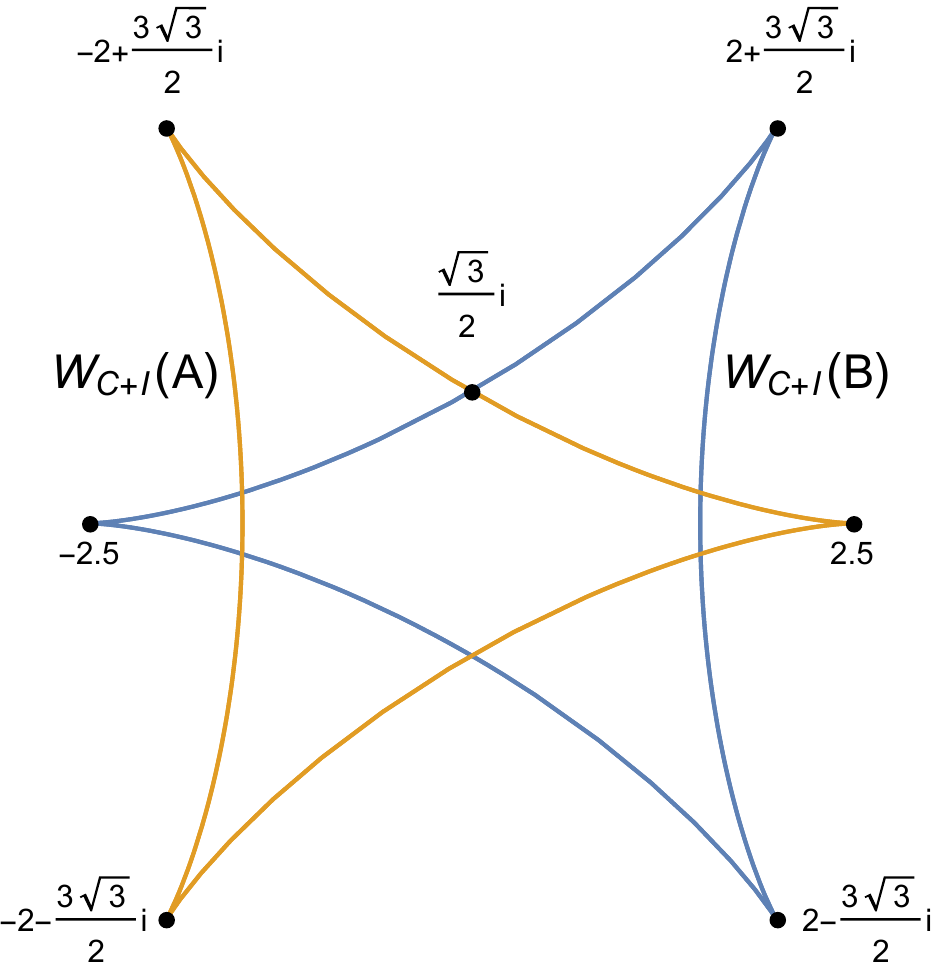} 
\caption{$W_{C+I}(A)\cup W_{C+I}(B)$}
\label{fig9}
\end{figure}

The ``vertices" of $\Gamma(0) $ and $\Gamma(1)$ are given by 
$$ 
\begin{array}{rclrclrcl}
f\(-\frac{2\pi}{3},0\) &= &  2  -\frac{3\sqrt{3}}{2}i&f(0,0) &= &  2 + \frac{3\sqrt{3}}{2}i&f\(\frac{2\pi}{3},0\) &= &  -\frac{5}{2}\\&\\
f\(-\frac{2\pi}{3},1\) &= &  -2  -\frac{3\sqrt{3}}{2}i&f(0,1) &= &  \frac{5}{2}&f\(\frac{2\pi}{3},1\) &= &  -2  + \frac{3\sqrt{3}}{2}i\,.\end{array}$$

As $t$ increases from $0$ to $1$, $\Gamma(t) $ changes from $W_{C+I}(B)$ to  $W_{C+I}(A)$, For $\theta=-\frac{2\pi}{3}, 0, \frac{2\pi}{3}$, the vertex $f(\theta,t)$ moves along the line segment from $f(\theta,0)$ to $f(\theta,1)$, see \Cref{fig7}.

\begin{figure}[H]
\centering
\includegraphics[height=5cm]{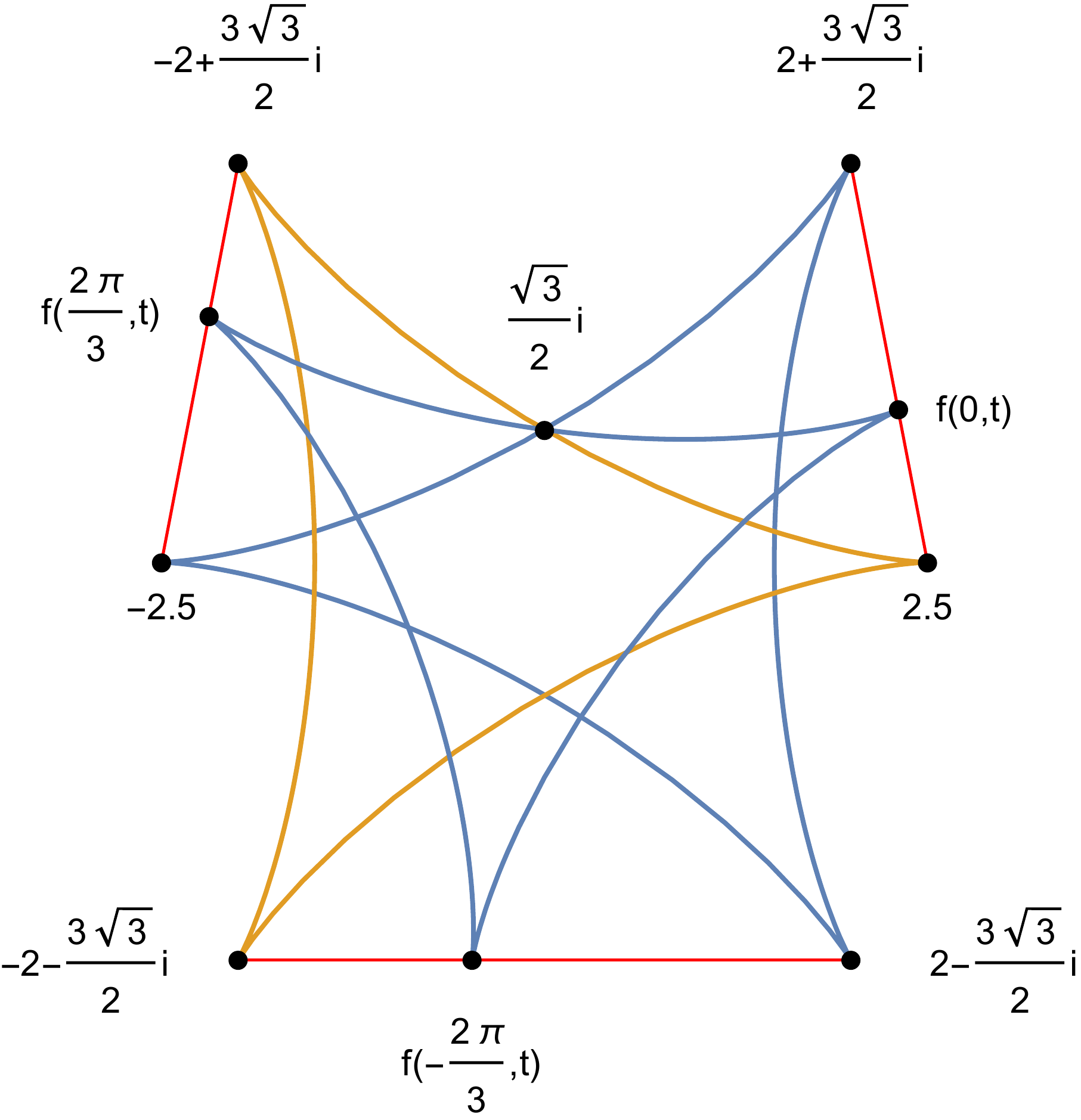} 
\caption{ $W_{C+I}(\cF)$}
\label{fig7}
\end{figure}

Therefore, we see that $W_{C+I}(\cF)$ is the union of $W_{C+I}(A)$, $W_{C+I}(B)$ and three triangles, $\triangle\(0.5,2.5,2+\frac{3\sqrt{3}}{2}i\)$,  $\triangle\(-0.5,-2.5,-2+\frac{3\sqrt{3}}{2}i\)$,   $\triangle\(\frac{\sqrt{3}}{2}i,2-\frac{3\sqrt{3}}{2}i,-2-\frac{3\sqrt{3}}{2}i\)$, see \Cref{fig6}.

\begin{figure}[H]
\centering
\includegraphics[height=5cm]{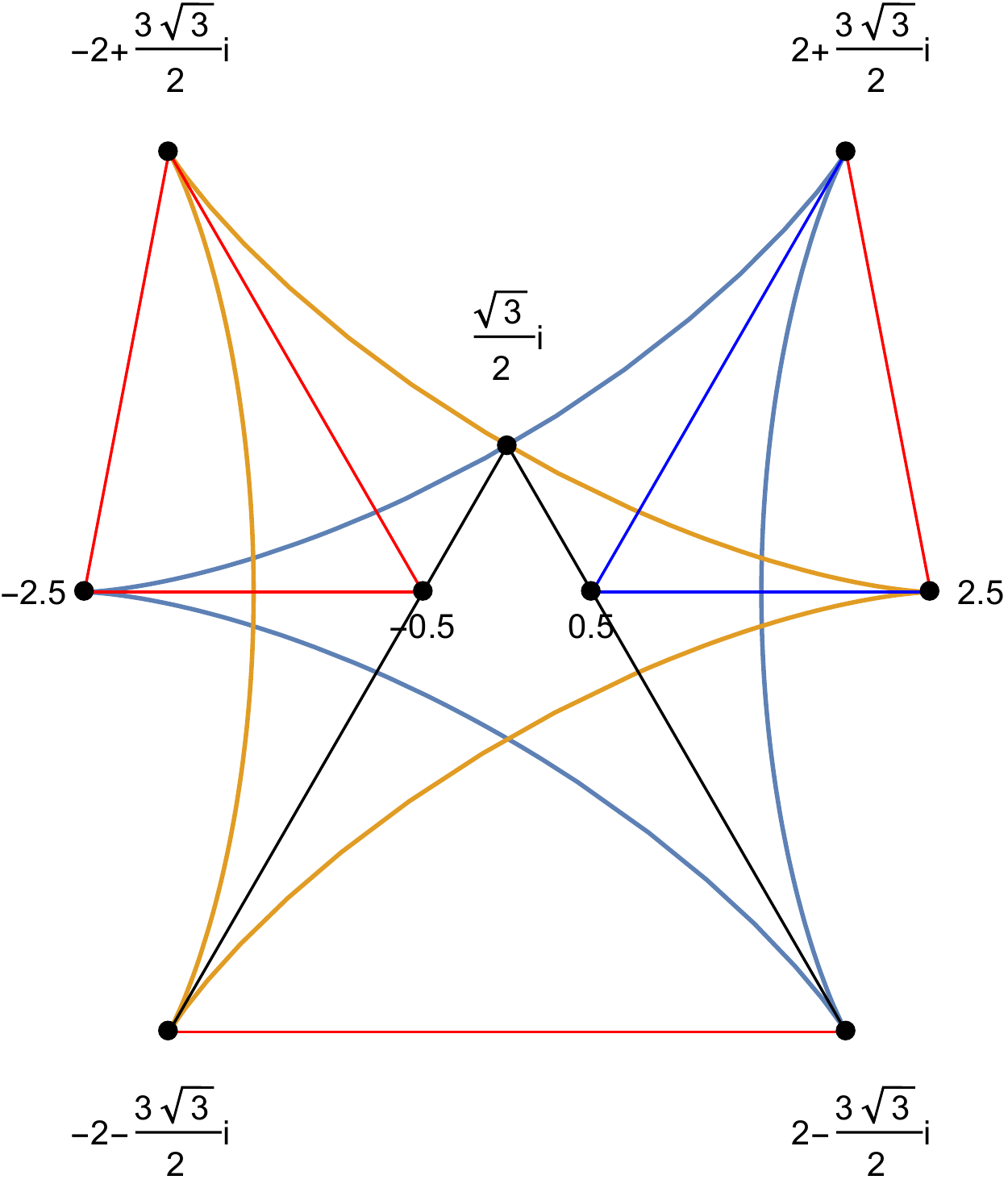} 
\caption{ $W_{C+I}(\cF)$}
\label{fig6}
\end{figure}

Suppose $ W_{C+I}(\cF)$ is star-shaped. By symmetry, $ W_{C+I}(\cF)$ has a star-center $c$ on the imaginary axis. The line
segment joining $c$ and $f(0,0)$ must lie below  the tangent line to $\Gamma(0)$ at $f(0,0)=2+\frac{3\sqrt{3}}{2}i$. By  direct calculation, this  tangent line is the line  through $0.5$ and  $f(0,0)$. 
Similarly,  The line
segment joining $c$ and $f(1,0)$ must lie above the tangent line to $\Gamma(1)$ at $f(1,0)=2.5$. By direct calculation, this  tangent line is the line  through $0.5$ and  $f(1,0)$. Since these two tangent lines intersect at $0.5$, no $c$ on the imaginary axis will satisfy the above conditions.

\end{example}

In the following, we present an example of $C, A_1, A_2, A_3 \in M_n$
such  that $W_C(A_1)$, $W_C(A_2)$, $W_C(A_3)$ are convex
and $W_C(A_1) \cap W_C(A_2) \cap W_C(A_3) = \emptyset$, but 
$W_C(\cF)$ is not star-shaped 
for $\cF = \conv\{A_1, A_2, A_3\}$.
In particular, we choose $C = E_{11}$ so that $W_C(\cF) = W(\cF)$. 

To describe our example, we first show that the set $W(\cF)$ is connected to
the concept of product numerical range arising in the study of 
quantum information theory, \cite{Karol}.

Let $A\in M_m\otimes M_n = M_m(M_n)$, the product numerical range of $A$ is 
$$W^{\otimes}(A)=\{(u\otimes v)^*A(u\otimes v):u\in \IC^m,\ v\in\IC^n\}.$$
We note that $W^{\otimes}(A)$ depends on order of the factors  $M_m$ and $ M_n$ 
in  the representation of $M_{mn}= M_m\otimes M_n$.

\medskip
\begin{theorem}\label{p} Let $A_1,\dots ,A_m\in M_n$ and $\cF=\conv(\{A_1,\dots ,A_m\})$. 
Then $$W(\cF)=W^{\otimes}(\oplus_{i=1}^mA_i).$$ 
\end{theorem}

\begin{proof} \rm Let $\mu\in W(\cF)$ Then there exist $t_1,\dots, t_m\ge 0$ with  $\sum_{i=1}^mt_i=1$ and $v\in \IC^n$ such that $\mu = v^*\(\sum_{i=1}^mt_iA_i\)v$. Let $u=(\sqrt{t_1},\dots, \sqrt{t_m})^t$. Then $\mu=(u\otimes v)^*(\oplus_{i=1}^mA_i)(u\otimes v)\in W^{\otimes}(\oplus_{i=1}^mA_i)$.

Conversely, suppose  $\mu\in  W^{\otimes}(\oplus_{i=1}^mA_i)$. Then there exist unit 
vectors $u\in \IC^m,\ v\in\IC^n$ such that 
$\mu=(u\otimes v)^*(\oplus_{i=1}^mA_i)(u\otimes v)$. Let $u=(u_1,\dots,u_m)^t$. 
Set $t_i=|u_i|^2$ for $i=1,\dots,m$. Then $\mu= v^*\(\sum_{i=1}^mt_iA_i\)v\in W(\cF)$.
\end{proof}

\medskip

Now, we describe the follow example showing that $W(\cF)$ is not necessarily star-shaped
for $\cF = \conv\{A_1, A_2, A_3\}$ if $W(A_1) \cap W(A_2) \cap W(A_3) = \emptyset$.

\medskip

\begin{example} \label{productrange}
\rm Suppose 
$A=\diag\(e^{i\frac{\pi}{3}},e^{-i\frac{\pi}{3}},0.95e^{i\frac{\pi}{4}}\)$. 
Let  $A_1=e^{i\frac{\pi}{3}}A,\ A_2=e^{-i\frac{\pi}{3}}A$ 
and $A_3=0.95e^{i\frac{\pi}{4}}A$. Then by \Cref{p}, 
$W(\conv\{A_1,A_2,A_3\})$ is equal to
$$W^{\otimes}(\oplus_{i=1}^mA_i)=W^{\otimes}(A\otimes A)=W(A)\cdot W(A)=\{\mu_1\mu_2:\mu_1,\mu_2\in W(A)\}.$$
By the result in \cite[Example 3.1]{LPW}, $W(A)\cdot W(A)$  is not star-shaped.

We note that in this example, $W(A_i)\cap W(A_j)\ne \emptyset$ for all $i,j$, but  
$W(A_1)\cap W(A_2)\cap W(A_3)= \emptyset$. Therefore, in some sense, the condition 
in \Cref{intersect_m} is optimal.
\end{example}

While \Cref{intersect_m} and \Cref{nonintersect} provide some
sufficient conditions for the star-shapedness of $W_C(\cF)$, it is a
challenging problem to determine whether $W_C(\cF)$ is star-shaped or not
for a given $C \in M_n$ and $\cF \subseteq M_n$.

\begin{proposition}\label{3.5}
Suppose
$\cF=\conv\{A_1,\dots, A_m\}\subseteq M_n$.
For each unit vector $x\in\IC^n$, let $\cA_x = \diag(x^*A_1x,\dots, x^*A_mx)$,
and let
$\hat\cF= \{\cA_x:x\in\IC^n,\ x^*x=1\}$. Then
$$W(\cF)=W(\hat\cF)
=W(\conv \hat \cF).$$
\end{proposition}

\begin{proof} Let $\mu \in W(\cF)$. Then there exist $t_1,\dots,t_m\ge 0$,
$t_1+\cdots +t_m=1$ and  a unit vector $x\in \IC^n$ such that
$\mu=x^*(t_1A_1+\cdots +t_mA_m)x= t_1x^*A_1x+\cdots +t_mx^*A_mx\in W(\hat\cF)$. Therefore, $W(\cF)\subseteq W(\hat\cF)$.

Clearly, $W(\hat \cF) \subseteq W(\conv \hat \cF)$.
Finally, we show that $W(\conv\hat \cF) \subseteq W(\cF)$.
To see this, let
$\mu \in W(\conv\hat\cF)$, i.e.,
$$\mu = y^*\left(\sum_{j=1}^r t_j D_j \right)y $$
for a unit vector
$y = (y_1, \dots, y_m)^t$,   some
$D_j = \diag (x_j^*A_1x_j, \dots, x_j^*A_mx_j)\in \hat \cF$
with  $j = 1, \dots, r$  and $t_1,\dots,t_r>0, \ t_1+\cdots + t_r=1$.
Thus,
\begin{eqnarray*}
\mu & = & \sum_{\ell=1}^m |y_\ell|^2
\sum_{j= 1}^r t_j x_j^* A_\ell  x_j 
= \sum_{j=1}^r t_j \left(x_j^*\left(\sum_{\ell=1}^m |y_\ell|^2 A_\ell \right)x_j\right) \\
&\in& \conv(W(A_0)) = W(A_0) \subseteq W(\cF)
\qquad
\end{eqnarray*}
where  $A_0 = \sum_{\ell=1}^m|y_\ell|^2A_\ell\in \cF$.
\end{proof}

\begin{example} \rm Suppose that $A_1=\diag\(1,e^{i\frac{\pi}{3}}\)$, $A_2=\diag\(e^{i\frac{2\pi}{3}},e^{i\pi}\)$, $A_2=\diag\(e^{i\frac{4\pi}{3}},e^{i\frac{5\pi}{3}}\)$ and $\cF=\conv\{A_1,A_2,A_3\}$. Then $ W(A_1)\cap W(A_2)\cap W(A_3)=\emptyset$. So the condition in \Cref{intersect_m} is not satisfied. However, by direct computation, we have
 $$\hat\cF=\left\{\diag\(t+(1-t)e^{i\frac{\pi}{3}},e^{i\frac{2\pi}{3}}( t+(1-t)e^{i\frac{\pi}{3}} ),e^{i \frac{4\pi}{3}} ( t+(1-t)e^{i\frac{\pi}{3}} )\):0\le t\le 1\right\}\,,$$
where $\hat\cF$ is defined as in \Cref{3.5}.
Therefore, $0\in \cap_{A\in \hat\cF} W(A)$. Hence, by \Cref{3.5} and \Cref{intersect_m}, $W(\cF)=W(\hat \cF)$ is star-shaped.
\end{example}

Recall that $W_C(\cF)$ may fail to be convex even if $\cF$ is convex and $W_C(A)$ is convex for all $A\in\cF$.
In the following, we give a necessary and sufficient condition for $W_C(\cF)$ to be convex. First, it is easy to see that $W_C(A) \subseteq W_C(\cF)$ for every $A \in \cF$.
Thus, $\conv \{\cup_{A \in \cF} W_C(A)\}$ is the smallest convex set containing
$W_C(\cF)$. As a result, we have the following observation.

\begin{proposition}\label{c1.2}
Let $C\in M_n$ and $\cF \subseteq M_n$. Then $W_C(\cF)$ is convex if and only if
$$W_C(\cF) = \conv\left\{\bigcup_{A \in \cF} W_C(A)\right\}.$$
\end{proposition}

\begin{proposition} \label{2.3.2}
Suppose $C\in M_n$, $\cG \subseteq M_n$ and $\cF = \conv\cG$ with $|\cG| \ge 3$.
Then
for every $\mu \in W_C(\cF)$,  there are matrices $A_1, A_2, A_3 \in \cG$
and a unitary matrix $U$ such that $\mu \in \conv\{\tr (CU^*A_1U),\tr (CU^*A_2U), \tr (CU^*A_3U)\}$.
If $\mu$ is a boundary point of $W_C(\cF)$, then
there are $B_1, B_2 \in \cG$
and a unitary $V$ such that $\mu \in \conv\{\tr (CV^*B_1V), \tr (CV^*B_2V)\}$.
\end{proposition}
\begin{proof}
For any $\mu \in W_C(\cF)$, we have
$\mu = t_1\tr (CU^*A_1U) + \cdots + t_r \tr (CU^*A_rU)$ for some
unitary matrix $U$ and $A_1, \dots, A_r \in \cG$ with $t_1, \dots, t_r > 0$ summing up
to 1. So, $\mu \in \conv\{\tr (CU^*A_1U), \dots, \tr (CU^*A_rU)\}$.
Thus, $\mu$ lies in the convex hull of no more than
three of the points in $\{\tr (CU^*A_1U), \dots, \tr (CU^*A_rU)\}$.
In case $\mu$ is a boundary point of $W_C(\cF)$, then
$\mu$ must be a boundary point of the set
$\conv\{\tr (CU^*A_1U), \dots, \tr (CU^*A_rU)\}$.
Thus, $\mu$ list in the convex hull of no more than two
points in the set  $\{\tr (CU^*A_1U), \dots, \tr (CU^*A_rU)\}$. 
The assertion follows. \end{proof}

\begin{theorem}\label{c2.2}
Suppose $C\in M_n$, $\cG\subseteq M_n$  is compact and $\cF=\co(\cG)$.
The following are equivalent.
\begin{enumerate}
\item[{\rm (a)}]  $W_C(\cF)$ is convex.
\item[{\rm (b)}] $W_C(\cF)$ is simply connected and
 every boundary point $\mu \in \conv W_C(\cF)$ has the form
$\tr (CU^*(tA + (1-t)B)U)$ for some unitary matrix $U$, $t \in [0,1]$ and
$A, B \in \cG$.
\end{enumerate}
\end{theorem}

\begin{proof} Suppose $W_C(\cF)$ is convex. Then it is clearly simply connected.
Now, $\conv W_C(\cF)$ and $W_C(\cF)$ have the same boundary. By
\Cref{2.3.2}, every boundary point of $W_C(\cF)$ is a convex combination of
$\tr (CU^*AU),\tr (CU^*BU) $ with $A,B \in \cG$.

Conversely, suppose (b) holds. We only need to show that each boundary
point $\mu$ of $\conv(W_C(\cF))$ lies in $W_C(\cF)$, which is true by
\Cref{2.3.2} and assumption (b). \end{proof}

\section{Extension to the joint $C$-numerical range}

Let $A = A_1+iA_2 \in M_n$, where $A_1, A_2 \in M_n$ are Hermitian matrices.
Then $W(A)$ can be identified as the joint numerical range of 
$(A_1, A_2)$ defined by
$$W(A_1,A_2) = \{ (x^*A_1x, x^*A_2x): x \in \IC^n, x^*x = 1\}.$$
One may consider whether our results can be extended to the joint 
numerical range of an $m$-tuple of Hermitian matrices $ (A_1, \dots, A_m)$
defined by
$$W(A_1, \dots, A_m) = \{(x^*A_1x, \dots, x^*A_mx): x \in \IC^n, x^*x = 1\}.$$

Some of the results on classical numerical range 
are not valid for the joint numerical range. 
For instance, the joint numerical range of three 
matrices may not be convex if $n = 2$.
For $A_1 = \begin{pmatrix} 0 & 1 \cr 1 & 0 \cr\end{pmatrix}, 
A_2 = \begin{pmatrix} 0 & -i \cr i & 0 \cr\end{pmatrix},
A_3 = \begin{pmatrix} 1 & 0 \cr 0 & -1 \cr\end{pmatrix},$ 
we have
$$W(A_1, A_2, A_3) = \{(a,b,c): a, b, c \in \IR, a^2+b^2 + c^2=1\}.$$
The following is known, see \cite{YHAuYeung1979,LP}.

\begin{enumerate}
\item Suppose $A_1, A_2, A_3\in M_2$ are Hermitian matrices such that 
$\{I_2, A_1, A_2, A_3\}$ is linearly independent.
Then $W(A_1,A_2,A_3)$ is an ellipsoid without interior in $\IR^3$.

\item  If $n\ge 3$ and $A_1, A_2, A_3 \in M_n$ 
are Hermitian matrices, then
$W(A_1,A_2,A_3)$ is convex.

\item Suppose $A_1, A_2, A_3 \in M_n$ such that
$\{I_2, A_1, A_2, A_3\}$ is linearly independent. Then there is 
$A_4 \in M_n$ such that $W(A_1, A_2, A_3, A_4)$ is not convex. 

\end{enumerate}

There has been study of topological and 
geometrical properties of $W(A_1, \dots, A_m)$. 
Researchers also consider the joint $C$-numerical range
of $(A_1, \dots, A_m)$ defined by 
\begin{equation}
\label{JWCA}
W_C(A_1, \dots, A_m) = \{ (\tr(CU^*A_1U), \dots, \tr(CU^*A_m U)): 
U \hbox{ unitary}\}
\end{equation}
for a Hermitian matrix $C$, for example \cite{Chien2013,MDChoi2003}. In particular, it is known that 
\break $W_C(A_1, A_2, A_3)$ is convex for any Hermitian matrices 
$C, A_1, A_2, A_3 \in M_n$ if $n \ge 3$, see~\cite{YHAuYeung1983b}.
Of course, one can also consider the $C$-numerical range of 
$(A_1, \dots, A_m)$ for general matrices $C, A_1, \dots, A_m \in M_n$
defined as in \Cref{JWCA}.

Denote by $\bA = (A_1, \dots, A_m)$ an $m$-tuple of matrices in 
$M_n$. Let $\bF$ be a non-empty subset of  $M_n^m$. We consider 
$$W_C(\bF) =  \bigcup \{W_C(\bA): \bA \in \bF \}.$$
Evidently, when $\bF = \{\bA\}$, then $W_C(\bF) = W_C(\bA)$.

\subsection{Basic results}

We begin with the following results.

\begin{proposition} \label{P3.1} Let
$C \in M_n$ be non-scalar, and let 
$\bF$ be a non-empty subset of $M^m_n$.

\begin{enumerate}
\item For any unitary $U, V \in M_n$, 
we have $W_C(\bF) =  W_{V^*CV}(U^*\bF U)$, where
$$U^*\bF  U = \{ (U^*A_1U, \dots, U^*A_mU): (A_1, \dots, A_m) \in \bF\}.$$

\item Let $\gamma_1, \gamma_2 \in \IC$ with $\gamma_1 \ne 0$.
If $\hat C = \gamma_1 C + \gamma_2 I_n$, then for any 
$\bA = (A_1, \dots, A_m) \in M_n^m$,
$$W_{\hat C}(\bA)
= \{ \gamma_1(a_1, \dots, a_m) + \gamma_2 (\tr A_1, \dots, \tr A_m):
(a_1,\dots, a_m) \in W_C(\bA)\}.$$

\item For any $T = (t_{ij}) \in M_m$ and $f = (f_1, \dots, f_m)^t$, we can define
an affine map $R$ on $\IC^m$ by $v \mapsto Tv + f$,
and  extend the affine map to $M_n^m$ by mapping $\bA=(A_1, \dots, A_m)$ to 
$\bB = (B_1, \dots, B_m)$ with 
$B_i = \sum_{j=1}^n t_{ij}A_j + f_i I_n$.
Then 
$$W_C(\bB)  = \{ (b_1, \dots, b_m): 
(b_1, \dots, b_m)^t = T(a_1, \dots, a_n)^t + \hskip .7in \ 
$$
$$
\hskip 1.5in (\tr C) (f_1, \dots, f_m)^t,
(a_1, \dots, a_m) \in W_C(\bA)\}.$$
Consequently,  
$$R(W_C(\bF)) = W_C(R(\bF)).$$
 
\item The linear span of $\{A_j-(\tr A_j)/n: j = 1, \dots, m\}$
has dimension $k$ if and only if $W_C(\bF) \subseteq \bV + f$
for a $k$-dimensional subspace $\bV \subseteq \IC^m$ and 
a vector $f \in \IC^m$. In particular, 
$W_C(\bF)$ is a singleton $\{(\nu_1, \dots, \nu_n)\}$ if and only if
$\cF = \{(\mu_1 I, \dots, \mu_m I)\}$ with 
$(\tr C)(\mu_1, \dots, \mu_m) = (\nu_1, \dots, \nu_n)$.

\end{enumerate}
\end{proposition}

Note that if $C \in M_n$ is Hermitian, then 
$W_C(A_1, \dots, A_m) \subseteq \IC^m$ can be identified as 
$W_C(X_1, Y_1, X_2, Y_2, \dots, X_m, Y_m) \subseteq \IR^{2m}$,
where  $X_j = (A_j+A_j^*)/2$ and $Y_j = (A_j-A_j^*)/(2i)$.
One can obtain a ``real version'' of \Cref{P3.1} using
real scalars $\gamma_1, \gamma_2$, real matrix $T$,  real vector $f$, etc.

One can prove the following when 
$W_C(\bF)$ is a polyhedral set in $\IC^m$, i.e., 
a convex combination of a finite set of vertices.

\begin{proposition} Suppose $C \in M_n$ is non-scalar, and 
$\bF$ is a non-empty set of $M_n^m$. If $W_C(\bF)$ is polyhedral,
then every vertex   has the form 
$(\mu_1, \dots, \mu_m)$ with $\mu_j = \tr V^*CV U^*A_jU$,
where $(A_1, \dots, A_m) \in \bF$,
$U, V \in M_n$ are unitary so that $V^*CV, U^*A_1U, \dots, U^*A_mU$
are in lower triangular matrices
with diagonal entries 
$$c_1, \dots, c_n, a_1(1), \dots, a_n(1),
\dots, a_1(m), \dots, a_n(m)$$ 
and $\mu_j=\sum_{\ell=1}^n c_\ell a_{\ell}(j)$.
Furthermore, if $c_1, \dots, c_n$ are distinct, then 
$$V^*CV, U^*A_1U,\dots, U^*A_mU$$ 
are diagonal matrices.
\end{proposition}

If $C$ has distinct eigenvalues, and if $W_C(\bA)$ 
has a conical point $(\mu_1, \dots, \mu_m)$ on the boundary, 
i.e., there is a pointed cone $K \subseteq \IC^m \equiv \IR^{2m}$ 
with vertex $(\mu_1, \dots, \mu_m)$ such that 
$$W_C(\bA) \cap \{(\mu_1 + \nu_1, \dots, \mu_m+\nu_m):
|\nu_j| \le \varepsilon\} \subseteq K$$
for some sufficiently small $\varepsilon > 0$, 
then $\{C, A_1, \dots, A_m\}$ is a commuting family of normal matrices,
and $\conv W_C(\bA)$ is a polyhedral set, see \cite{Binding}.

Note that if there is $\bB \in \bF$ such that
$W_C(\bB)$ is a singleton, a subset of a straight line, or a convex polygon,
then we can apply the results on $W_C(\bB)$ to deduce that $C$ and $\bB$ has 
special structure. Then we can deduce results on $W_C(\bF)$.

We can obtain some topological properties $W_C(\bF)$.

\begin{proposition}
Let $C \in M_n$ be non-scalar, and $\bF\subseteq M_n^m$ be a nonempty set.
\begin{enumerate}
\item If $\cF$ is bounded, then so is $W_C(\bF)$.
\item If $\cF$ is connected, then so is $W_C(\bF)$.
\item If $\cF$ is compact, then so is $W_C(\bF)$.
\end{enumerate}
\end{proposition}

\begin{proof} The proof of the boundedness and compactness are similar to 
those for $W_C(\cF)$ in Section 2. The prove connectedness, 
for any $\bA, \bB \in \bF$ and unitary $U_0, U_1\in M_n$, there is a path
joining $U_t$ with $t \in [0,1]$ joining $U_0$ and $U_t$,
and hence 
there is a path joining $(\tr U_0^*CU_0 A_1, \dots, \tr U_0^*CU_0A_m)$
to $(\tr U_1^*CU_1 A_1, \dots, \tr U_1^*CU_1A_m)$, which is connected to
$(\tr U_1^*CU_1 B_1, \dots, \tr U_1^*CU_1B_m)$.
\end{proof}

\subsection{Star-shapedness and convexity}

Here we consider whether $W(\bF)$ is star-shaped or  convex. It is known that if $C\in M_n$ is Hermitian, then $W_C(A_1,A_2,A_3)$ is convex for any Hermitian $A_1, A_2, A_3 \in M_n$ with $n \ge 3$, see~\cite{YHAuYeung1983b}.
One may wonder whether 
\Cref{intersect_m} admits an extension to this setting.
The following example shows that the answer is negative.

\begin{example} Let  
$$\bA_t = \left(
\begin{pmatrix} t & & \cr & t & \cr & & 1-t \cr\end{pmatrix},
\begin{pmatrix} 1-t & & \cr & t & \cr & & 1-t\cr\end{pmatrix}
\begin{pmatrix} 0 & & \cr & t & \cr & & 1-t \cr\end{pmatrix} \right),
\qquad t \in [0, 1].$$
Notice that $\co\{\bA_0,\bA_1\}=\{\bA_t:0\le t\le 1 \}$. Moreover, $(0,0,0) \in W(\bA_0 )\cap W( \bA_1)$, but $(0,0,0)$ 
is not a star-center of $W(\conv\{ \bA_0, \bA_1\})$ as
$(1/2,1/2,0) \in W(\bA_{1/2})$, however,  $\conv\{(0,0,0), (1/2,1/2,0)\} 
\not \subseteq W(\conv\{\bA_0, \bA_1\})$. Nevertheless, $(1/2,1/2,1/2)\in W(\bA_t)$
for all $t \in [0,1]$, and hence $(1/2,1/2,1/2)$ is a star-center of 
$W(\conv\{\bA_0, \bA_1\})$.
Actually, one can show that $W(\bA_t) = \co\{(t,1-t,0), (t,t,t), (1-t,1-t,1-t)\}.$
\end{example}

Here is an example showing that $W(\bA)$ may not be star-shaped in general.

\begin{example} 
Let $\bA=(A_1,A_2,A_3)$ with $A_1=\diag(0,1,0)$, $A_2=\diag(1,0,-1)$, $A_3=I_3$ and let $\bB=(B_1,B_2,B_3)$ with $B_1=\diag(1,0,0),$ $B_2=\diag(0,-1,1)$, $B_3=0_3$. If $\bF=\co\{\bA,\bB\}$, then $W(\bF)$ is the union of the triangular disk with vertices
$$(1-t,t,t),(t,t-1,t),(0,1-2t,t).$$
Let $g:\IR^3\to \IR^3$ be given by $g((a,b,c))= (a,-b,1-c)$. We have 
$$\begin{array}{rll}g((1-t,t,t))&=(1-t,-t, 1-t)&=(1-t,(1-t)-1,1-t)\\&\\
g((t,t-1,t))&=(t,1-t,1-t)&=(1-(1-t),(1-t),(1-t))\\&\\
g((0,1-2t,t))&=(0,2t-1,1-t)&=(0,1-2(1-t),(1-t))\end{array}$$
Therefore, $g(W(\bF))=W(\bF)$.
Moreover, we claim that $W(\bF)$ is not star-shaped.

\rm  
Suppose the contrary that $W(\bF)$ is star-shaped with $(a,b,t_0)$ be a star-center. Then $(a,-b,1-t_0)$ is also a star-center. 
As the set of all star-center of a star-shaped set is convex. We may now assume without loss of generality that $b=0$ and $t_0=1/2$. Assume now $a>0$. By assumption, for all $0\leq t\leq 1$, 
$t(a,0,1/2)+(1-t)(1,0,0)=(1-t(1-a),0,t/2)\in W(\bF).$
By direct computation, given $0\leq t\leq 1$, we have
$$\max\{ \alpha\in\mathbb{R}:(\alpha,0,t)\in W(\bF)\}= t^2+(1-t)^2=1-2t+2t^2.$$
However for sufficient small $t>0$
$$1-t(1-a)>1-t(1-t^2/2)>(t/2)^2+(1-t/2)^2,$$
which contradicts that $(1-t(1-a),0,t/2)\in W(\bF)$. Therefore, we have $a=0$. However $(0,0,1/2)$ is not a star-center as 
$$\frac{1}{3}(0,0,1/2)+\frac{2}{3}(0,0,1)=\left(0,0,\frac{5}{6}\right)\notin\co\left\{\left(\frac{1}{6},\frac{5}{6},\frac{5}{6}\right),\left(\frac{5}{6},-\frac{1}{6},\frac{5}{6}\right),\left(0,-\frac{2}{3},\frac{5}{6}\right)\right\}.$$
Therefore, $W(\bF)$ is not star-shaped. 
\end{example}

It is challenging to determine conditions on $\bF$ so that $W(\bF)$ is 
star-shaped.

\medskip
\noindent
{\large\bf Acknowledgment}

Our study was inspired by some discussion at the AIM workshop on
Crouzeix's conjecture, July 30 - August 5, 2017, AIM, San Jose.
The second and the third authors would like to express their thanks
to the organizer and the colleagues at AIM for the well-organized 
and stimulating workshop.

Li is an honorary professor of the Shanghai University,
and an affiliate member of the Institute for Quantum Computing, University of
Waterloo; his research was supported by
the USA NSF DMS 1331021, the Simons Foundation
Grant 351047, and NNSF of China Grant 11571220.
Research of Sze and Lau were supported by a PolyU central research grant G-YBKR and a HK RGC grant PolyU 502512.
The HK RGC grant also supported the post-doctoral fellowship of Lau at the Hong Kong Polytechnic University.

\noindent
{\bf Addresses}

\noindent
\noindent(P.S. Lau) Department of Applied Mathematics, The Hong Kong Polytechnic University,
Hung Hom, Hong Kong. (panshun.lau@polyu.edu.hk)

\noindent(C.K. Li)
Department of Mathematics, College of William \& Mary,
Williamsburg, VA 23185.
(ckli@math.wm.edu)

\noindent
(Y.T. Poon)
Department of Mathematics, Iowa State University,
Ames, IA 50011.
\\ (ytpoon@iastate.edu)

\noindent(N.S. Sze) Department of Applied Mathematics, The Hong Kong Polytechnic University,
Hung Hom, Hong Kong. (raymond.sze@polyu.edu.hk)

\end{document}